\documentclass[12pt,oneside,english]{amsart}
\usepackage[T1]{fontenc}
\usepackage[latin9]{inputenc}
\usepackage[letterpaper]{geometry}
\usepackage{units}
\usepackage{amsthm}
\usepackage{amssymb}

\numberwithin{equation}{section} 
\numberwithin{figure}{section} 
\theoremstyle{plain}
 \theoremstyle{definition}
 \newtheorem*{defn*}{Definition}
\theoremstyle{plain}
\newtheorem{thm}{Theorem}
  \theoremstyle{plain}
  \newtheorem*{thm*}{Theorem}
  \theoremstyle{plain}
  \newtheorem*{cor*}{Corollary}
  \theoremstyle{plain}
  \newtheorem*{lem*}{Lemma}
  \theoremstyle{plain}
  \newtheorem{cor}[thm]{Corollary}
  \theoremstyle{remark}
  \newtheorem*{claim*}{Claim}

\usepackage{babel}

\begin{document}

\title{Groups of $p$-Deficiency One}

\author{Anitha Thillaisundaram}

\date{1st August 2011}
\begin{abstract}
The main result of \cite{key-But} is that all finitely presented
groups of $p$-deficiency greater than one are $p$-large. Here we
prove that groups with a finite presentation of $p$-deficiency one
possess a finite index subgroup that surjects onto $\mathbb{Z}$.
This implies that these groups do not have Kazhdan's property (T).
Additionally, we prove that the main result of \cite{key-But} implies
a result of Lackenby \cite{key-Lac1}.
\end{abstract}
\maketitle

\section*{Introduction}

This paper continues from the author's joint paper with Button \cite{key-But}. 

Throughout this paper, $p$ denotes a prime. Recall that the $p$-deficiency
of a group is defined as follows.
\begin{defn*}
\cite{key-Profi,key-SchPuc2} Let $G$ be a finitely generated group.
Say $G\cong\langle X|R\rangle$, with $|X|$ finite. For a prime $p$,
the \emph{$p$-deficiency} of $G$ with presentation $\langle X|R\rangle$
is\[
\text{def}_{p}(G;X,R)=|X|-\sum_{r\in R}p^{-\nu_{p}(r)},\]
where $\nu_{p}(r)=\max\left\{ k\ge0\ |\ \exists w\in F(X),\ w^{p^{k}}=r\right\} $.
The $p$-deficiency of $G$ is then defined to be the supremum of
$\text{def}_{p}(G;X,R)$ over all presentations $\langle X|R\rangle$
of $G$ with $|X|$ finite.
\end{defn*}
This is similar to the deficiency of a group.
\begin{defn*}
The \emph{deficiency} of a group $G$ is \[
\text{def}(G)=\sup_{\langle X|R\rangle}\{|X|-|R|:G\cong\langle X|R\rangle\}.\]

\end{defn*}
We recall the concepts of largeness and $p$-largeness.
\begin{defn*}
\cite{key-Lac2} Let $G$ be a group, and let $p$ be a prime. Then 

$\bullet$ $G$ is \emph{large} if some (not necessarily normal) subgroup
with finite index admits a non-abelian free quotient; 

$\bullet$ $G$ is \emph{$p$-large} if some normal subgroup with
index a power of $p$ admits a non-abelian free quotient.
\end{defn*}
\pagebreak{}

The main result of \cite{key-But} is the following.
\begin{thm}
(\cite{key-But}, Theorem 2.2) \label{thm:Main} Let $p$ be a prime.
If $G$ is a finitely presented group with $p$-deficiency greater
than one, then $G$ is $p$-large.
\end{thm}
Theorem \ref{thm:Main} is proved using results of Lackenby (\cite{key-Lac2},
Theorem 1.15) and Schlage-Puchta \cite{key-Profi,key-SchPuc2}. Please
see \cite{key-But} for details of the proof.

By Corollary 2.1 of \cite{key-But}, groups with a finite presentation
of $p$-deficiency one are infinite. In this paper, we prove the following.
\begin{thm*}
Let $\Gamma$ be a finitely presented group with a presentation of
$p$-deficiency equal to one, for some prime $p$. Then $\Gamma$
has a finite index subgroup $H$ that surjects onto $\mathbb{Z}$.
\end{thm*}
The Related Burnside Problem (\cite{key-Khu}, Problem 8.52) asks
whether or not there exist infinite finitely presented torsion groups.
The theorem above extends Corollary 2.4 of \cite{key-But} to give
the following response to the Related Burnside Problem.
\begin{cor*}
Let $G$ be an infinite finitely presented group with a presentation
of $p$-deficiency greater than or equal to one, for some prime $p$.
Then $G$ is not torsion.
\end{cor*}
We note here the definition of Kazhdan's property (T). Please see
\cite{key-Bekka} for more information of property (T).
\begin{defn*}
\cite{key-Bekka} Let $\Gamma$ be a finitely generated group. 

(a) Given a unitary representation $V$ of $\Gamma$ and a generating
set $S$ of $\Gamma$, we define $\kappa(\Gamma;S;V)$ to be the largest
$\varepsilon\ge0$ such that for any $v\in V$ there exists $s\in S$
with $||sv-v||\ge\varepsilon||v||$. 

(b) Given a generating set $S$ of $\Gamma$, the \emph{Kazhdan constant}
$\kappa(\Gamma;S)$ is defined to be the infimum of the set $\left\{ \kappa(\Gamma;S;V)\right\} $
where $V$ runs over all unitary representations of $\Gamma$ without
non-zero invariant vectors. 

(c) The group $\Gamma$ is called a \emph{Kazhdan group} (equivalently
$\Gamma$ is said to have \emph{Kazhdan's property (T)}) if $\kappa(\Gamma;S)>0$
for some (hence any) finite generating set $S$ of $\Gamma$.
\end{defn*}
It is proved in (\cite{key-Bekka}, Corollaries 1.3.6 \& 1.7.2) that
if a group has property (T), then its finite index subgroups must
have finite abelianization.

Section 1 of this paper presents the proof of our main result. The
corollaries of our main result are the content of Section 2. We finish
with Section 3 which gives an interesting example of a group with
a finite presentation of 3-deficiency one, and we comment on Ershov's
finitely presented Golod-Shafarevich group with property (T).

This paper is mostly an extract, which was under the supervision of
Jack Button, of the author's PhD thesis. 

The author is grateful to the Cambridge Commonwealth Trust, the Cambridge
Overseas Research Scholarship, and the Leslie Wilson Scholarship (from
Magdalene College) for their financial support.

$\ $

\section{Main Result}

There is an amount of ambiguity in saying that a group is of $p$-deficiency
one. Formally, the $p$-deficiency of a group is the supremum of $\text{def}_{p}(\langle X|R\rangle)$
over all presentations $\langle X|R\rangle$ of the group. Therefore
it is theoretically possible for the $p$-deficiency of a group to
be one in the limit, but with none of $\text{def}_{p}(\langle X|R\rangle)$
being equal to one.

We avoid this delicate situation by insisting that the group has a
presentation of $p$-deficiency one. This is the convention that we
adopt whenever we deal with $p$-deficiency one groups.

First, we note a result from \cite{key-Allc}.
\begin{thm}
\label{thm:Allc}\cite{key-Allc} Let $G$ be a group with presentation\[
\langle a_{1},\ldots,a_{n}|1=w_{1}^{r_{1}}=\ldots=w_{m}^{r_{m}}\rangle\]
where each $w_{j}$ is a word in the $a_{i}$ and their inverses.
Suppose that $H$ is a normal subgroup of $G$ of index $N<\infty$
and that for each $j$, $w_{j}^{k}\not\in H$ for $k=1,\ldots,r_{j}-1$.
Then the rank of the abelianization of $H$ is at least\[
1+N\left(n-1+\sum_{i}\frac{1}{r_{i}}\right).\]

\end{thm}
We now prove the following.
\begin{thm}
\label{thm:ourpdef1} Let $\Gamma$ be a finitely presented group
with a presentation of $p$-deficiency equal to one, for some prime
$p$. Then $\Gamma$ has a finite index subgroup $H$ that surjects
onto $\mathbb{Z}$.\end{thm}
\begin{proof}
Suppose $\Gamma$ is a finitely presented group with a presentation
of $p$-deficiency equal to one. Thus we have \[
\Gamma\cong\langle x_{1},\ldots,x_{d}|w_{1},\ldots,w_{r},w_{r+1}^{p^{a_{r+1}}},\ldots,w_{q}^{p^{a_{q}}}\rangle\]
with \[
\text{def}_{p}(\Gamma)=d-r-\sum_{i=r+1}^{q}\frac{1}{p^{a_{i}}}=1\]
where $d,q\in\mathbb{N}$, $0\le r\le q$, and $a_{r+1}\le\ldots\le a_{q}$
are positive integers.

$\ $

We now consider the following two cases which depend on the orders
of the $w_{i}$'s.

$\ $

\emph{\underbar{Case (i).}}

For $i=1,\ldots,q$, if the $w_{i}$'s have exact orders as in the
presentation above, then for some finite index normal subgroup $H$
of $\Gamma$ we have that the rank of the abelianization of $H$ is
at least one by Theorem \ref{thm:Allc}. Hence $H$ surjects onto
$\mathbb{Z}$, as required.

$\ $

\emph{\underbar{Case (ii).}}

Now without loss of generality suppose that for all normal subgroups
$K$ of finite index in $\Gamma$, we have $w_{q}^{p^{a_{q}-1}}\in K$.

Consider now \[
G=\langle x_{1},\ldots,x_{d}|w_{1},\ldots,w_{r},w_{r+1}^{p^{a_{r+1}}},\ldots,w_{q-1}^{p^{a_{q-1}}},w_{q}^{p^{a_{q}+1}}\rangle.\]
Note that \[
\Gamma\cong\nicefrac{G}{\langle\langle w_{q}^{p^{a_{q}}}\rangle\rangle}\]
and as $\text{def}_{p}(G)>1$ we have that $G$ is $p$-large.

Recall the definition of $p$-large. Let $H$ be a normal subgroup
in $G$ of index $p^{k}$, for $k\ge0$, such that there exists a
surjection $\psi:H\twoheadrightarrow F_{2}$.

For ease of notation, we write $w:=w_{q}$ and $a:=a_{q}$. We may
assume that the order of $w$ in $G$ is $p^{a+1}$, as if $o(w)<p^{a+1}$,
then $\Gamma=G$ is $p$-large, and we are done.

Now, our plan is to consider $\nicefrac{G}{\langle\langle w^{p^{a}}\rangle\rangle}$
and show that this quotient group has a finite index subgroup that
surjects onto $\mathbb{Z}$. As $\Gamma\cong\nicefrac{G}{\langle\langle w^{p^{a}}\rangle\rangle}$,
this completes the proof of our theorem.

$\ $

Consider the order of $\overline{w}$, the image of $w$ in $\nicefrac{G}{H}$.

$\ $

a) If $o(\overline{w})$ in $\nicefrac{G}{H}$ is $\le p^{a}$, then
this implies that $w^{p^{a}}\in H$. We will show that $\nicefrac{G}{\langle\langle w^{p^{a}}\rangle\rangle}$
is $p$-large to obtain our contradiction.

Let $k_{1},\ldots,k_{s}$ be a set of representatives for the cosets
of $H$ in $G$. Let $m\ (\le p^{a})$ be the smallest positive integer
such that $k_{j}w^{m}k_{j}^{-1}\in H$, for each $j$. Note that $m$
divides $p^{k}=[G:H]$. Let $n$ be any positive integer, and let
$G_{mn}=\langle\langle w^{mn}\rangle\rangle$ be the subgroup of $G$
generated normally by $w^{mn}$. Note that this is contained in $H$,
and is in fact the subgroup of $H$ normally generated by $\{k_{j}w^{mn}k_{j}^{-1}:1\le j\le s\}$.
Now $\{\psi(k_{j}w^{mn}k_{j}^{-1}):1\le j\le s\}$ is a collection
of elements in $F_{2}$. The key thing to note here, is that $k_{j}w^{mn}k_{j}^{-1}$
all have orders a power of $p$ in $H$, and so their images under
$\psi$ must be trivial in $F_{2}$. So $G_{mn}\le\ker\psi$, and
we have the induced surjection $\overline{\psi}:\nicefrac{H}{G_{mn}}\twoheadrightarrow F_{2}$.
Now $\nicefrac{H}{G_{mn}}$ has finite $p^{\text{th}}$ power index
in $\nicefrac{G}{G_{mn}}$. Therefore $\nicefrac{G}{G_{mn}}$ is $p$-large.
Finally, we take $mn=p^{a}$, and therefore $\nicefrac{G}{\langle\langle w^{p^{a}}\rangle\rangle}\cong\Gamma$
is $p$-large. The result now follows for $\Gamma$.

$\ $

b) If $o(\overline{w})$ in $\nicefrac{G}{H}$ is $p^{a+1}$, then
in $\nicefrac{G}{H\langle\langle w^{p^{a}}\rangle\rangle}$ the image
of $w$ has order dividing $p^{a}$. As this is a finite $p$-quotient
of $\Gamma$, we use the fact that $w$ has, in fact, exact order
$p^{a}$ in $\nicefrac{G}{H\langle\langle w^{p^{a}}\rangle\rangle}$,
and so we are back in Case (i).
\end{proof}
The last line of the proof draws on the following simple fact from
finite $p$-groups.
\begin{lem*}
Let $g$ be an element of a finite $p$-group $G$, and say $o(g)=p^{k}$
for $k>0$. Let $N=\langle\langle g^{p^{k-1}}\rangle\rangle$. Then
$g^{p^{k-2}}\not\in N$.\end{lem*}
\begin{proof}
We consider the Frattini subgroup $\Phi(N)$ of $N$, which is defined
to be the intersection of all maximal subgroups of $N$. It is well-known
that $\Phi(N)$ is characteristic in $N$, and that $\Phi(N)=N^{p}N'$
as $N$ is a finite $p$-group. As $\Phi(N)$ is characteristic in
$N$, and $N$ is normal in $G$, we have that $\Phi(N)$ is normal
in $G$.

Firstly, we note that we cannot have $g^{p^{k-1}}$ belonging to $\Phi(N)$:
if $g^{p^{k-1}}\in\Phi(N)$, then all conjugates $h^{-1}g^{p^{k-1}}h$,
for $h\in G$, also lie in $\Phi(N)$. This means that $N=\Phi(N)$,
which is impossible.

Now suppose that $g^{p^{k-2}}\in N$. Then \[
(g^{p^{k-2}})^{p}=g^{p^{k-1}}\in N^{p}\le\Phi(N),\]
a contradiction. Thus $g^{p^{k-2}}\not\in N$, as required.
\end{proof}
$\ $

\section{Corollaries}

Theorem \ref{thm:ourpdef1} together with Corollary 2.4 of \cite{key-But}
imply the following response to the Related Burnside Problem.
\begin{cor}
Let $G$ be an infinite finitely presented group with a presentation
of $p$-deficiency greater than or equal to one, for some prime $p$.
Then $G$ is not torsion.
\end{cor}
The next corollary incorporates Kazhdan's property (T).
\begin{cor}
\label{cor:Prop(T)} Let $G$ be an infinite finitely presented group
with a presentation of $p$-deficiency greater than or equal to one,
for some prime $p$. Then $G$ does not have property (T).\end{cor}
\begin{proof}
By Theorems \ref{thm:Main} and \ref{thm:ourpdef1}, we know that
$G$ has a finite index subgroup $H$ such that $H$ surjects onto
$\mathbb{Z}$. The result now follows from Corollary 1.3.6 and Corollary
1.7.2 of \cite{key-Bekka}.
\end{proof}
$\ $

Furthermore, we have the following result from \cite{key-Lac1}. Case
(iia) of the proof of Theorem \ref{thm:ourpdef1} follows the proof
of Theorem \ref{thm:(,-Theorem-1.1)} closely. 
\begin{thm}
\label{thm:(,-Theorem-1.1)}\cite{key-Lac1} Let $G$ be a finitely
generated, large group and let $g_{1},\ldots,g_{r}$ be a collection
of elements of $G$. Then for infinitely many integers $n$, $\nicefrac{G}{\langle\langle g_{1}^{n},\ldots,g_{r}^{n}\rangle\rangle}$
is also large. Indeed, this is true when $n$ is any sufficiently
large multiple of $[G:H]$, where $H$ is any finite index normal
subgroup of $G$ that admits a surjective homomorphism onto a non-abelian
free group.
\end{thm}
Below is a stronger statement for free groups which is used in the
proof of Theorem \ref{thm:(,-Theorem-1.1)}.
\begin{thm}
\label{thm:Lem28} \cite{key-Lac1} Let $F$ be a finitely generated,
non-abelian free group. Let $g_{1},\ldots,g_{r}$ be a collection
of elements of $F$. Then, for all but finitely many integers $n$,
the quotient $\nicefrac{F}{\langle\langle g_{1}^{n},\ldots,g_{r}^{n}\rangle\rangle}$
is large.
\end{thm}
The above theorem has a topological proof. As in the proof of the
Nielsen-Schreier Theorem on subgroups of free groups, $F$ here is
viewed as the fundamental group of a bouquet of circles. Then the
quotient $\nicefrac{F}{\langle\langle g_{1}^{n},\ldots,g_{r}^{n}\rangle\rangle}$
is obtained by attaching 2-cells representing $g_{1}^{n},\ldots,g_{r}^{n}$
along the circles. More details are to be found in \cite{key-Lac1}.

$\ $

Olshanskii and Osin give a shorter algebraic proof of Theorem \ref{thm:(,-Theorem-1.1)}
in \cite{key-Ols}. The main body of Olshanskii and Osin's proof relies
on Theorem \ref{thm:Lem28}, which they also prove with alternative
algebraic arguments.

We remark here that Theorem \ref{thm:Lem28} (and hence Theorem \ref{thm:(,-Theorem-1.1)})
follows from Theorem \ref{thm:Main}. We remind the reader that Theorem
\ref{thm:Main} relies on another result of Lackenby (\cite{key-Lac2},
Theorem 1.15).
\begin{cor}
\label{cor:lackcor} Let $F$ be a free group of rank $r\ge2$, $g_{1},\ldots,g_{k}$
arbitrary elements of $F$. Then $\overline{F}\cong\nicefrac{F}{\langle\langle g_{1}^{q},\ldots,g_{k}^{q}\rangle\rangle}$
is large for all but finitely many $q\in\mathbb{N}$.\end{cor}
\begin{proof}
We consider the $p$-deficiency of $\overline{F}$:\[
\text{def}_{p}(\overline{F})\ge r-\frac{k}{p^{l_{p}}},\]
where $p$ is some prime factor of $q$, and $l_{p}$ is the highest
power of $p$ dividing $q$. By Theorem \ref{thm:Main}, $\overline{F}$
is $p$-large if $\text{def}_{p}(\overline{F})>1$, that is, when
$p^{l_{p}}>\frac{k}{r-1}$. 

So as long as $p^{l_{p}}>\frac{k}{r-1}$ for at least one $p$ dividing
$q$, then $\overline{F}$ is large. That is, for all but finitely
many $q\in\mathbb{N}$, $\overline{F}$ is large.
\end{proof}
$\ $

Lackenby's proof of Theorem \ref{thm:Lem28} (or Corollary \ref{cor:lackcor})
relies on topological arguments, which span over a few pages. Theorem
\ref{thm:Main} has enabled us to present a short proof of a different
spirit. 

$\ $

\section{Examples}

Clearly finitely presented groups of $p$-deficiency one exist and
examples include the infinite dihedral group $D_{\infty}=\langle x_{1},x_{2}|x_{1}^{2},x_{2}^{2}\rangle$,
all groups of deficiency one, and the group $P=\langle x,y,z|x^{3},y^{3},z^{3},(xy)^{3},(xz)^{3},(yz)^{3}\rangle$.
The group $D_{\infty}$ is not torsion nor large. The groups of deficiency
one are not torsion but some are large (see \cite{key-Bu1}). The
group $P$ was verified by MAGMA to be 3-large (and hence is not torsion).
For the group $P$, we used the approach shown below, as is similar
to Subsection 4.2 of \cite{key-But}.
\begin{claim*}
The group $P=\langle x,y,z|x^{3},y^{3},z^{3},(xy)^{3},(xz)^{3},(yz)^{3}\rangle$
is 3-large. \end{claim*}
\begin{proof}
Using MAGMA's LowIndexNormalSubgroups function, we considered the
following index three normal subgroup of $P$:\[
C=\langle a,b,c,d|[c^{-1},a^{-1}],[d^{-1},b^{-1}],adc^{-1}a^{-1}bcd^{-1}b^{-1}\rangle,\]
which was (at the time of writing) seventh on the list of fourteen
normal subgroups with index at most three in $P$. The above presentation
for $C$ was obtained using MAGMA's Simplify function.

Then we formed the quotient \[
\nicefrac{C}{\langle\langle c,d\rangle\rangle}\]
and we noticed that the quotient is isomorphic to $\langle a,b\rangle\cong F_{2}$.
Hence $C$ is 3-large by definition, and since $C$ is normal in $P$
of index $3$, we have proved that $P$ is 3-large, as required.
\end{proof}
$\ $

With reference to Corollary \ref{cor:Prop(T)}, the following example
is due to Ershov \cite{key-Spain}. Let $d\ge9$ and $p<(d-1)^{2}$,
then the group\[
G\cong\langle x_{1},\ldots,x_{d}|\left[[x_{i},x_{j}],x_{j}\right]=1\ \forall i\ne j,\ x_{i}^{p}=1\rangle\]
is a finitely presented Golod-Shafarevich group with property (T)
(see \cite{key-Ers} for further information). Naturally the $p$-deficiency
of $G$ is not one.

\end{document}